\documentclass{article}
\usepackage{times,amsmath,amsthm,graphicx,epstopdf}
\newtheorem{lemma}{Lemma}
\newtheorem{theorem}{Theorem}
\newtheorem*{proposition}{Kose-Uzawa Theorem}

\newtheorem{corollary}{Corollary}

\newcommand{\half}{\frac{1}{2}}

\renewcommand{\vec}[1]{{\bf #1}}

\DeclareMathOperator*{\argmax}{argmax}

\title{Convergence of gradient descent-ascent analyzed as a Newtonian dynamical   system with dissipation}

\author{H. Sebastian Seung\thanks{Present address: Neuroscience Institute and Computer Science Department, Princeton University. {\tt sseung@princeton.edu}}
\\Howard Hughes Medical Institute and\\
  Dept. of Brain \& Cog. Sci. and Dept. of Physics, MIT}

\begin{document}

\maketitle

\begin{abstract}
  A dynamical system is defined in terms of the gradient of a payoff function.   Dynamical variables are of two types, ascent and descent.  The ascent   variables move in the direction of the gradient, while the descent variables   move in the opposite direction. Dynamical systems of this form or very   similar forms have been studied in diverse fields such as game theory,   optimization, neural networks, and population biology. Gradient   descent-ascent is approximated as a Newtonian dynamical system that conserves   total energy, defined as the sum of the kinetic energy and a potential energy   that is proportional to the payoff function. The error of the approximation   is a residual force that violates energy conservation. If the residual force   is purely dissipative, then the energy serves as a Lyapunov function, and   convergence of bounded trajectories to steady states is guaranteed.  A   previous convergence theorem due to Kose and Uzawa required the payoff   function to be convex in the descent variables, and concave in the ascent   variables. Here the assumption is relaxed, so that the payoff function need   only be globally ``less convex'' or ``more concave'' in the ascent variables   than in the descent variables. Such relative convexity conditions allow the existence of multiple steady states, unlike the convex-concave assumption. When combined with sufficient conditions that   imply the existence of a minimax equilibrium, boundedness of trajectories is   also assured.\footnote{This paper was completed in 2007 but never published. It is being posted in its original form, without updating the references. The hope is that the work might be useful to others, even though it is possible that the results are no longer novel.} 
\end{abstract}

In the dynamical system
\begin{subequations}\label{eq:GradientComponents}
\begin{align}
\frac{dx_i}{dt} & = -\frac{\partial S}{\partial x_i}, \qquad  i=1,\ldots,m \\
\frac{dy_j}{dt} & = \frac{\partial S}{\partial y_j}, \qquad j=1,\ldots,n
\end{align}
\end{subequations}
the $x$ variables move down the gradient of $S$, while the $y$ variables move up the gradient.  Therefore the dynamics will be called {\em gradient   descent-ascent} on the function $S(x_1,\ldots,x_m,y_1,\ldots,y_n)$, which will be assumed twice differentiable. Because of the applications to game theory described later, $S$ will be called the {\em payoff function} of the dynamical system.

For the special cases of pure gradient ascent ($m=0$) and pure gradient descent ($n=0$), it is well-known that bounded trajectories typically converge to steady states. However, for a mixture of gradient descent and ascent, the dynamical behavior depends strongly on the properties of $S$. For example, suppose that $m=n=1$. If $S=(x^2-y^2)/2$, then the dynamics converges to the origin. But if $S=xy$, all nonzero trajectories are periodic orbits.  Although these two functions differ only by a 45 degree rotation in the $xy$ plane, they yield very different dynamical behaviors. 

This paper establishes sufficient conditions for convergence to steady states by constructing a Lyapunov function for gradient descent-ascent.  The convergence theorems are of interest, because they are more powerful than a previous theorem due to Kose \cite{kose1956ssv} and Uzawa \cite{Uzawa58}.  But the methods of proof may actually be more interesting than the theorems themselves.

While the theorems are specific to dynamical systems that can be written exactly as in Eqs.\ \eqref{eq:GradientComponents}, the methods of proof may generalize to other dynamical systems that only resemble gradient descent-ascent.  As will be explained below, such dynamical systems are studied in fields as diverse as game theory \cite{arrow1960sgp,singh2000ncg}, constrained optimization \cite{ArrowHurwicz58gmc1}, 
neural networks \cite{mjolsness1990ato,seung1998mah,gao2004nnc},
and population biology \cite{kerner1997chs,samuelson1974bla}.

Once the Lyapunov function is in hand, the proofs are fairly straightforward. But it may not be obvious how to construct the function in the first place.  It turns out that the Lyapunov function emerges naturally when gradient descent-ascent is approximated as a Newtonian dynamical system that conserves total energy. The approximation is not perfect, as there is a residual force that violates energy conservation. If the residual force is purely dissipative or frictional, the total energy of the system is nonincreasing with time, and serves as a Lyapunov function. This is the case in which convergence can be proven. Otherwise the residual force may both add and remove energy, so that convergence may not occur.

A short review of Lyapunov functions will be helpful\cite{logemann2004abn}. Let $\vec{z}(t)$ be the trajectory of a set of differential equations, and let $L(\vec{z})$ be a function with continuous partial derivatives. If for any trajectory $L(\vec{z}(t))$ is nonincreasing with time, and constant only at steady states of the dynamics, then $L$ is called a Lyapunov function.\footnote{The term Lyapunov-like function might be more accurate, as    Lyapunov's original definition contains more restrictive   conditions.\cite{slotine1991anc}} 
If the steady states of the dynamics are isolated points, then the existence of a Lyapunov function implies that any bounded trajectory converges to a steady state.  \footnote{ This is a corollary   of LaSalle's Invariance Principle\cite{lasalle1960eas}.  Without the   assumption that the steady states are isolated, convergence to a steady state   is not guaranteed. For example, pathological examples are known in which a   pure gradient dynamics approaches a manifold of steady states, but never   converges to a single steady state\cite{curry1944msd, absil16cid}.} Furthermore, if the set $\{\vec{z}|L(\vec{z})\leq c\}$ is bounded for all $c$, then boundedness of trajectories follows. This condition is guaranteed, for example, if $L$ is radially unbounded, i.e., $L(\vec{z})\to\infty$ as $|\vec{z}|\to\infty$.

\section{Applications of gradient descent-ascent}
Dynamical systems that are gradient descent-ascent, or are very similar, have been studied in the fields of game theory, optimization, neural networks, and population biology. 

\subsection{Game theory}
Equations \eqref{eq:GradientComponents} were proposed by Arrow and Hurwicz as a simple model for how agents change their actions when playing a game repeatedly \cite{arrow1960sgp,singh2000ncg}. One can think of the function $S$ as defining a zero-sum game with continuous action spaces. Suppose that $\vec{x}\in R^m$ and $\vec{y}\in R^n$ are actions chosen by two agents, and $S(\vec{x},\vec{y})$ is the payoff from one to the other.  The payer tries to decrease $S$ by gradient descent, while the payee tries to increase $S$ by gradient ascent. 

The replicator dynamics is a popular model for learning matrix games. It can be viewed as a modification of gradient descent-ascent that respects the constraints that $\vec{x}$ and $\vec{y}$ are normalized probability vectors. Hofbauer has observed that the replicator dynamics has a Hamiltonian structure \cite{hofbauer1996edb}, which is related to the present approach of approximating gradient descent-ascent by a conservative Newtonian dynamics.

\subsection{Optimization}
Arrow and Hurwicz proposed gradient descent-ascent as an algorithm for convex optimization with nonnegativity constraints \cite{ArrowHurwicz58gmc1}. The payoff function $S$ is the Lagrangian, the $x_i$ are the primal variables, and the $y_j$ are the Lagrange multipliers or dual variables. In this application, Eqs.\ \eqref{eq:GradientComponents} are modified to respect the nonnegativity constraints, a trick known as gradient projection.  The theorems proven in this paper can be extended to this modification of gradient descent-ascent.


\subsection{Neural networks}
While gradient descent-ascent may not be the most efficient optimization algorithm for software implementation on digital computers, it is still an excellent algorithm for direct hardware implementation.  This was recognized by Kose, who described an analog electronic computer that implemented gradient descent-ascent in his original paper that also contained  a theoretical analysis of convergence \cite{kose1956ssv}.  More recently, artificial neural networks have been constructed for solving optimization problems, and optimization has been proposed as a computational function of biological neural networks \cite{hopfield1986cnc}.  It has been noted that certain neural network models are very similar in form to gradient descent-ascent \cite{mjolsness1990ato,Platt87}. The convergence of such networks can be analyzed using the ideas that are discussed in this paper \cite{seung1998mah}.

\subsection{Population biology}
The Lotka-Volterra equations were proposed as a model for population dynamics. Suppose that the different species in the model can be separated into two groups. Interactions within each group are symmetric, while interaction between the groups are antisymmetric. Then the equations can be viewed as a version of gradient descent-ascent, modified to respect the constraints that population sizes are nonnegative. Volterra originally considered the case of purely antisymmetric interactions, for which the dynamics takes a Hamiltonian form \cite{kerner1997chs,samuelson1974bla}.

\section{Informal summary of convergence results}
In this section, the results of the paper are summarized informally, and compared with previous work.  Fifty years ago, Kose \cite{kose1956ssv} and Uzawa \cite{Uzawa58} considered the case where $S$ is a convex-concave function. 
If $S$ is strictly convex in the descent variables $\vec{x}$ and strictly concave in the ascent variables $\vec{y}$, then gradient descent-ascent converges to a saddle point, if one exists. 
Of the two examples given in the introduction, $S=(x^2-y^2)/2$ is strictly convex-concave, while $S=xy$ is not.  This paper establishes two new convergence theorems based on substantially weaker assumptions. In Theorem 1,
the convex-concave assumption is replaced by
\begin{quote}
1. $S$ is ``globally less convex'' in $\vec{y}$ than in $\vec{x}$ 
(or equivalently ``globally more concave'').
\end{quote}
The meaning of the phrase ``globally less convex'' will be made precise later using the Hessian of $S$.  Assumption 1 is sufficient to guarantee that bounded trajectories converge to steady states.  This means that periodic orbits are impossible; trajectories either converge to steady states, or diverge to infinity.

To exclude the possibility that trajectories diverge, Assumption 1 must be augmented by some other condition. Theorem 2 guarantees boundedness of trajectories by adding one or the other of the following assumptions.
\begin{quote}
2. $U(\vec{x})\equiv \max_\vec{y}S(\vec{x},\vec{y})$ is radially unbounded in $\vec{x}$. 
\end{quote}
\begin{quote}
3. $-V(\vec{y})\equiv -\min_\vec{x}S(\vec{x},\vec{y})$ is radially unbounded in $\vec{y}$. 
\end{quote}
A function $f(\vec{x})$ is defined to be radially unbounded if $f\to\infty$ as $|\vec{x}|\to\infty$ in any direction.  The theorem is split into two cases, and one of these assumptions is used for each case.  
To understand the implications of Assumptions 2 and 3, note that a radially unbounded function always attains its minimal value at some point.
Assumption 2 guarantees that $\min_\vec{x}\max_\vec{y}S(\vec{x},\vec{y})$ exists, while Assumption 3 guarantees that $\max_\vec{y}\min_\vec{x}S(\vec{x},\vec{y})$ exists.  In other words, these assumptions guarantee the existence of minimax equilibria. \footnote{A minimax equilibrium must be a stationary point of $S$ and a steady   state of the dynamics, but the converse is not necessarily true.}

\section{Newtonian dynamics}
At this point, some readers may prefer to skip to the formal statements of the above theorems, and their proofs. Others may find the following physical analogy helpful, because it gives meaning to the Lyapunov function used in the proofs.  Below it will be shown that gradient descent-ascent can be approximated by a Newtonian dynamical system that conserves energy. The error of the approximation is a residual force that violates energy conservation.  From this physical viewpoint, the ``globally less convex'' condition mentioned above makes the residual force frictional, so that it cannot increase the energy of the system. 

Define $S_{\vec{x}}$ as the gradient of $S$ with respect to $\vec{x}$, i.e., $(S_{\vec{x}})_i = \partial S/\partial x_i$, and define $S_{\vec{y}}$ similarly.  Then the gradient descent-ascent equations \eqref{eq:GradientComponents} take the compact form
$\dot{\vec{x}}=-S_{\vec{x}}$ and $\dot{\vec{y}}=S_{\vec{y}}$.
Differentiation with respect to time yields
\begin{equation}\label{eq:ForceDecomposition}
\begin{pmatrix}
\ddot{\vec{x}}\\
\ddot{\vec{y}}
\end{pmatrix}
=
\begin{pmatrix}
-S_{\vec{x}\vec{x}}&  - S_{\vec{x}\vec{y}}\\
S_{\vec{y}\vec{x}}&  S_{\vec{y}\vec{y}}
\end{pmatrix}
\begin{pmatrix}
\dot{\vec{x}}\\
\dot{\vec{y}}
\end{pmatrix}
\end{equation}
where $S_{\vec{x}\vec{x}}$ is the matrix of partial derivatives of $S$ with respect to the $x$ variables, i.e., $(S_{\vec{x}\vec{x}})_{ij}=\partial^2S/\partial x_i\partial x_j$, and so on. Defining $\vec{z}=(\vec{x},\vec{y})$, this can be written in the form
\begin{equation}\label{eq:Newton}
\ddot{\vec{z}} = -\Phi_\vec{z} -K_A(\vec{z})\dot{\vec{z}} -K_S(\vec{z})\dot{\vec{z}} 
\end{equation}
Here $\Phi_\vec{z}$ is the gradient of the function $\Phi(\vec{z})=-rS(\vec{x},\vec{y})$, where $r$ is a scalar parameter.
The matrices $K_A$ and $K_S$ are defined by
\begin{equation}
K_A(\vec{z})=
\begin{pmatrix}
0 &  S_{\vec{x}\vec{y}}\\
- S_{\vec{y}\vec{x}}& 0
\end{pmatrix}
\qquad
K_S(\vec{z})=
\begin{pmatrix}
S_{\vec{x}\vec{x}}-rI& 0\\
0 & -(S_{\vec{y}\vec{y}}-rI)
\end{pmatrix}
\end{equation}
According to Newton's Second Law, the acceleration of an object is proportional to the force exerted on it. Therefore the three terms on the right hand side of Eq.\ \eqref{eq:Newton} can be interpreted as forces.  The first force could be called ``electric,'' as it depends on the gradient of the potential function $\Phi$.  The second force is {\em perpendicular} to velocity, because the matrix $K_A$ is {\em antisymmetric}, $K_A^T=-K_A$.  This term is analogous to the magnetic force on a charged particle, which is also perpendicular to velocity.  If these two forces (``electric'' and ``magnetic'') were the only ones that appeared in the dynamics, it would be a Newtonian system with the familiar property of energy conservation. The total energy is the sum of the kinetic energy $|\dot{\vec{z}}|^2/2$ and the potential energy $\Phi(\vec{z})$, or
\begin{equation}\label{eq:TotalEnergy}
L=\half|\dot{\vec{x}}|^2 + \half|\dot{\vec{y}}|^2 -rS(\vec{x},\vec{y})
\end{equation}
Because of the property of energy conservation, once set in motion a Newtonian dynamical system typically does not come to rest.
\footnote{This approach is a generalization of the method used by Platt and Barr for the special case of a payoff function that is linear in $\vec{y}$, which arises in constrained optimization \cite{Platt87}. They used the Lyapunov function \eqref{eq:TotalEnergy} with $r=0$.}

The third force in Eq.\ \eqref{eq:Newton} violates energy conservation. It will be called the ``residual,'' as it is the error of approximating gradient descent-ascent by a conservative Newtonian system.  A straightforward calculation shows that $\dot{L} = -\dot{\vec{z}}^T K_S\dot{\vec{z}}$.  Whether the residual force tends to make the total energy increase or decrease depends on the matrix $K_S$. For example, if $K_S$ is positive definite, then the total energy will be nonincreasing with time.  In this case, the residual force can be regarded as frictional, because it dissipates energy. When friction is added to a Newtonian system, it does come to rest. Note that the matrix $K_S$ depends on the scalar parameter $r$. As will be described below, the ``globally less convex'' condition guarantees that $K_S$ is positive definite for some choice of $r$.

\section{Extremal eigenvalues}
The ``globally less convex'' condition will now be made precise using two Hessian matrices.  The matrix $S_{\vec{x}\vec{x}}$ consists of partial derivatives of $S$ with respect to the $x$ variables, i.e., $(S_{\vec{x}\vec{x}})_{ij}=\partial^2S/\partial x_i\partial x_j$. The Hessian $S_{\vec{y}\vec{y}}$ is defined similarly.  It will be useful to summarize these Hessians by two extremal eigenvalues:
\begin{subequations}\label{eq:InfSupDef}
\begin{align}
\lambda_{inf}(S_{\vec{x}\vec{x}}) &= \inf_{\vec{x},\vec{y}}\lambda_{min}(S_{\vec{x}\vec{x}}(\vec{x},\vec{y}))\\
\lambda_{sup}(S_{\vec{y}\vec{y}}) &= \sup_{\vec{x},\vec{y}}\lambda_{max}(S_{\vec{y}\vec{y}}(\vec{x},\vec{y}))
\end{align}
\end{subequations}
Here $\lambda_{min}(S_{\vec{x}\vec{x}}(\vec{x},\vec{y}))$ is defined as the smallest eigenvalue of $S_{\vec{x}\vec{x}}$ at the point $(\vec{x},\vec{y})$, while $\lambda_{max}(S_{\vec{y}\vec{y}}(\vec{x},\vec{y}))$ is the largest eigenvalue of $S_{\vec{y}\vec{y}}$ at the point $(\vec{x},\vec{y})$.
The infimum and supremum are taken over the domain $R^m\times R^n$.

These numbers are related to the concepts of convexity and concavity.  If $\lambda_{inf}(S_{\vec{x}\vec{x}})\geq 0$, then $S_{\vec{x}\vec{x}}$ is positive semidefinite everywhere. This implies that $S$ is convex in $\vec{x}$ for any fixed $\vec{y}$, by the second-order condition for convexity. \footnote{If $\lambda_{inf}(S_{\vec{x}\vec{x}}) > 0$ (strictly positive), then   $S_{\vec{x}\vec{x}}$ is strictly positive definite everywhere, and $S$ is   strictly convex in $\vec{x}$ for any fixed $\vec{y}$.  If   $\lambda_{inf}(S_{\vec{x}\vec{x}}) = 0$, $S$ might or might not be strictly   convex in $\vec{x}$, depending on whether the infimum is actually attained at   some point.}  Similarly, if $\lambda_{sup}(S_{\vec{y}\vec{y}})\leq 0$, then $S_{\vec{y}\vec{y}}$ is negative semidefinite everywhere, and $S$ is concave in $\vec{y}$ for any fixed $\vec{x}$.

\section{Convergence of bounded trajectories}
Using the extremal eigenvalues (\ref{eq:InfSupDef}), the following sufficient condition for convergence to steady states can be formulated.
\begin{theorem}\label{thm:BoundedTrajectory}
Suppose that the steady states of the gradient descent-ascent dynamics \eqref{eq:GradientComponents} are isolated, and
\begin{equation}\label{eq:Convergence}
\lambda_{sup}(S_{\vec{y}\vec{y}}) < \lambda_{inf}(S_{\vec{x}\vec{x}})~.
\end{equation}
Then any bounded trajectory converges to a steady state.
\end{theorem}

The inequality \eqref{eq:Convergence} is the precise definition of the phrase ``globally less convex.''\footnote{To be clear, this inequality is a stronger condition than the statement that $S$ is {\em locally} less convex in $\vec{y}$ than in $\vec{x}$, or $\lambda_{max}(S_{\vec{y}\vec{y}}(\vec{x},\vec{y})) < \lambda_{min}(S_{\vec{x}\vec{x}}(\vec{x},\vec{y}))$ for all $(\vec{x},\vec{y})$.}
Note that Theorem \ref{thm:BoundedTrajectory} completely excludes the possibility of periodic orbits. The only way that a trajectory can fail to converge to a steady state is by diverging to infinity.

\begin{proof}
  The time derivative of the kinetic energy $T=   |\dot{\vec{x}}|^2/2+|\dot{\vec{y}}|^2/2$  is   $\dot{T}=-\dot{\vec{x}}^TS_{\vec{x}\vec{x}}\dot{\vec{x}}+   \dot{\vec{y}}^TS_{\vec{y}\vec{y}}\dot{\vec{y}}$, while the time derivative of   the payoff function is $\dot{S}=-|\dot{\vec{x}}|^2 +|\dot{\vec{y}}|^2$.   Define the function $L=T-rS$, which can be regarded as the sum of kinetic   energy $T$ and ``potential energy'' $-rS$ (see Eq.\ \eqref{eq:TotalEnergy} for more on the physical interpretation).  Its time derivative is
$\dot{L}=
-\dot{\vec{x}}^T\left(S_{\vec{x}\vec{x}}-rI\right)\dot{\vec{x}} +
\dot{\vec{y}}^T\left(S_{\vec{y}\vec{y}}-rI\right)\dot{\vec{y}}~.
$
Choose $r$ so that $\lambda_{sup}(S_{\vec{y}\vec{y}}) < r < \lambda_{inf}(S_{\vec{x}\vec{x}})$.  By definition, $S_{\vec{x}\vec{x}}-rI$ is positive definite everywhere while $S_{\vec{y}\vec{y}}-rI$ is negative definite everywhere.\footnote{Eq.\ (\ref{eq:Convergence}) could be replaced by   the slightly weaker condition that such an $r$ exists.  } It follows that $\dot{L}\leq 0$, with equality only at steady states of the dynamics, so $L$ is a Lyapunov function and convergence of bounded trajectories to steady states follows.
\end{proof}

\section{The convex-concave case}
Here Theorem \ref{thm:BoundedTrajectory} is compared with previous results concerning convergence of gradient descent-ascent.  The following corollary is useful for comparison.
\begin{corollary}\label{cor:Kinetic}
Suppose that the steady states of the gradient descent-ascent dynamics 
 \eqref{eq:GradientComponents} are isolated, and
\begin{equation}\label{eq:ConvexConcave}
\lambda_{sup}(S_{\vec{y}\vec{y}})<0<\lambda_{inf}(S_{\vec{x}\vec{x}}).
\end{equation}
Then any bounded trajectory converges to a steady state.
\end{corollary}
This special case of Theorem \ref{thm:BoundedTrajectory} is very similar to the following theorem proved by
Kose\cite{kose1956ssv} and Uzawa\cite{Uzawa58}.\footnote{Actually they considered a modification of gradient descent-ascent that respects nonnegativity constraints.}
\begin{proposition}
\label{thm:Uzawa}
  Suppose that $S$ is strictly convex in $\vec{x}$ for any fixed $\vec{y}$ and   strictly concave in $\vec{y}$ for any fixed $\vec{x}$.
\footnote{In fact, Uzawa only required weak concavity in $\vec{y}$, because he was interested in the case where $S$ is linear in $\vec{y}$, which arises in Lagrangian methods for optimization. This complicates the theorem somewhat.}
  Then gradient   descent-ascent converges to a saddle point if one exists.
\end{proposition}
A {\em saddle point} $(\vec{x}^\ast,\vec{y}^\ast)$ is defined as a point at which the inequality
$
S(\vec{x}^\ast,\vec{y})\leq S(\vec{x}^\ast,\vec{y}^\ast)\leq S(\vec{x},\vec{y}^\ast)
$
holds for all $\vec{x}$ and $\vec{y}$.  For a strictly convex-concave function $S$, it can be shown that there is at most one saddle point.
Furthermore, a saddle point is equivalent to a stationary point of $S$, or a steady state of gradient descent-ascent.

\begin{proof}[Proof sketch]
The  first-order condition for convexity implies that the Euclidean distance from   the saddle point is a Lyapunov function.  This implies that any bounded   trajectory converges to the saddle point.  Furthermore the Euclidean distance   is radially unbounded, so all trajectories are bounded.
\end{proof}

The Kose-Uzawa Theorem is similar to Corollary \ref{cor:Kinetic}, because condition \eqref{eq:ConvexConcave} implies that $S$ is strictly convex-concave.  However, the converse does not hold, so the Corollary is slightly weaker than the Theorem.  Also, the payoff function has to be twice differentiable in the Corollary, but only once differentiable in the Theorem.

The Theorem requires the assumption that a saddle point exists, in order to use the Euclidean distance from the saddle point as the Lyapunov function. Then all trajectories are guaranteed to be bounded, since the distance is radially unbounded.  On the other hand, Corollary \ref{cor:Kinetic} requires the assumption of a bounded trajectory, instead of the assumption that a saddle point exists. However, in the next section it will be shown that the assumption of a bounded trajectory actually follows from Eq.\ \eqref{eq:ConvexConcave}.

The Lyapunov function for Corollary 1 is simply the kinetic energy ($r=0$ in the proof of Theorem 1). With the convex-concave assumption, the kinetic energy is nonincreasing. Without the assumption, the kinetic energy may behave nonmonotonically with time.  This failure is connected with the possibility of multiple steady states, which can exist if the convex-concave condition is violated. Even if the system eventually converges to one state, it may approach other steady states closely before it finally converges, slowing down and speeding up each time. This is why it was necessary to combine the kinetic energy with the payoff function in Theorem 1 to produce a Lyapunov function that is nonincreasing.

\section{Boundedness of trajectories}
Since Corollary \ref{cor:Kinetic} is basically equivalent to the older Kose-Uzawa Theorem, the novel and interesting case is where condition \eqref{eq:ConvexConcave} is violated, but condition \eqref{eq:Convergence} is true.  Then the extremal eigenvalues have the same sign. For example, if both are positive, Eq.\ \eqref{eq:Convergence} means that $S$ is ``less convex'' in $\vec{y}$ than in $\vec{x}$. If the extremal eigenvalues are negative, then Eq.\ \eqref{eq:Convergence} means that $S$ is ``more concave'' in $\vec{y}$ than in $\vec{x}$.

Theorem \ref{thm:BoundedTrajectory} and Corollary \ref{cor:Kinetic} establish convergence of bounded trajectories to steady states. It remains to be determined whether trajectories are bounded. This can be proven 
after imposing additional conditions on $S$, by showing that the Lyapunov function is radially unbounded.  This is not obvious, as $L=T-rS$ is a combination of the kinetic energy and the payoff function.  While the kinetic energy is lower bounded by zero, the payoff function might diverge to $+\infty$ in some directions, and $-\infty$ in others.

\begin{theorem}\label{thm:MainTheorem}
  Suppose that   $\lambda_{sup}(S_{\vec{y}\vec{y}})<\lambda_{inf}(S_{\vec{x}\vec{x}})$.  If   either
\begin{enumerate}
\item 
$\lambda_{inf}(S_{\vec{x}\vec{x}})>0$ and 
$-V(\vec{y})=-\min_\vec{x} S(\vec{x},\vec{y})$ is radially unbounded, or
\item $\lambda_{sup}(S_{\vec{y}\vec{y}})<0$ and 
$U(\vec{x})=\max_{\vec{y}}S(\vec{x},\vec{y})$ is radially unbounded
\end{enumerate}
is also satisfied, then any trajectory of gradient descent-ascent is bounded.
\end{theorem}
Note that the first condition implies the existence of a minimax equilibrium, $\max_{\vec{y}}\min_{\vec{x}}S(\vec{x},\vec{y})$.  This follows because $\lambda_{inf}(S_{\vec{x}\vec{x}})>0$ implies the existence of $V(\vec{x})=\min_{\vec{x}}S(\vec{x},\vec{y})$ (see Lemma \ref{lem:QuadraticBound}, which guarantees that $S\to\infty$ as $|\vec{x}|\to\infty$.  The minimum is unique by strict convexity of $S$ in $\vec{x}$.  The existence of a maximum of $V$ follows from the fact that $-V$ is radially unbounded, but the maximum need not be unique.  Similarly, the second condition of the theorem guarantees the existence of a minimax equilibrium
$\min_{\vec{x}}\max_{\vec{y}}S(\vec{x},\vec{y})$.
Theorem \ref{thm:MainTheorem} has the following corollary. 
\begin{corollary}\label{cor:Kinetic2}
  Suppose that   $\lambda_{sup}(S_{\vec{y}\vec{y}})<0<\lambda_{inf}(S_{\vec{x}\vec{x}})$. Then   any trajectory of gradient descent-ascent is bounded.
\end{corollary}

\section{Two examples}
Two examples illustrate the power of the new theorems.  For the quadratic payoff function $S=ax^2/2+bxy+cy^2/2$, gradient descent-ascent is a linear dynamical system. By eigenvalue analysis, it converges to the origin when $c<a$ and $b^2>ac$. Theorem \ref{thm:MainTheorem} gives precisely the same conditions.  In contrast, the Kose-Uzawa Theorem guarantees convergence when $c<0<a$, which is much more restrictive.

The dynamics $\dot{x} = y-f(x)$ and $\dot{y}=-x+g(y)$ is gradient descent-ascent on the payoff function $S = F(x)-xy+G(y)$, where $F'(x)=f(x)$ and $G'(y)=g(y)$.  This reduces to the linear dynamics considered above if $f$ and $g$ are linear, but is nonlinear otherwise.  Since $S_{xx}=f'(x)$ and $S_{yy}=g'(y)$, by Theorem \ref{thm:BoundedTrajectory} any bounded trajectory converges to a steady state if\footnote{Since essentially the same condition   can be obtained from Bendixson's negative criterion   \cite{guckenheimer2002nod}, Theorem \ref{thm:BoundedTrajectory} is not really   needed for this simple example. But Theorem \ref{thm:BoundedTrajectory} is   applicable in higher dimensions, while Bendixson's criterion is only   applicable to two dimensions.}  $\inf_x f'(x) > \sup_y g'(y)$.  For example, suppose that $f(x) = \mu(x^3/3-x)$ and $g(y) = - \alpha y$, where $\mu$ and $\alpha$ are nonnegative parameters.  
In the special case $\alpha=0$ this is equivalent to the van der Pol oscillator, which has a limit cycle and an unstable steady state at the origin. The limit cycle persists for small nonzero $\alpha$.  If $\alpha>\mu$, there is no limit cycle by Theorem 1.
Furthermore, Theorem \ref{thm:MainTheorem} can be used to prove that all trajectories are bounded.  Since both $\inf_x f'(x)=-\mu$ and $\sup_y g'(y)=-\alpha$ are negative, the convex-concave assumption and the Kose-Uzawa theorem are not relevant.

\section{Proving boundedness of trajectories}

Theorem \ref{thm:MainTheorem} and Corollary \ref{cor:Kinetic2} will be proven by establishing radial unboundedness of the Lyapunov function. For this purpose, it is necessary to have inequalities that characterize the global behavior of the payoff function and its derivatives using the extremal eigenvalues \eqref{eq:InfSupDef}.  The first lemma is related to Taylor's theorem, which locally approximates a function using its first derivative and Hessian. The lemma gives a global lower bound for a function using its first derivative and extremal eigenvalue of its Hessian.
\begin{lemma}\label{lem:QuadraticBound}
Suppose that $\lambda_{inf}(f_{\vec{x}\vec{x}})$ exists
(see Eqs.   \eqref{eq:InfSupDef} for definition). Then $f(\vec{x})$ has the quadratic lower bound
$f(\vec{x}) \geq
f(\vec{a}) + f_{\vec{x}}(\vec{a})^T(\vec{x}-\vec{a})
+ \lambda_{inf}(f_{\vec{x}\vec{x}})|\vec{x}-\vec{a}|^2/2$.
\end{lemma}
\begin{proof}
The function $\bar{f}(\vec{x}) = f(\vec{x}) - \lambda_{inf}(f_{\vec{x}\vec{x}})|\vec{x}-\vec{a}|^2 /2$
is convex, because its Hessian $\bar{f}_{\vec{x}\vec{x}}=f_{\vec{x}\vec{x}}-\lambda_{inf}(f_{\vec{x}\vec{x}})I$
is negative semidefinite everywhere. Apply the first-order condition for convexity to $\bar{f}$.
\end{proof}
The lemma has a simple consequence if $\lambda_{inf}(f_{\vec{x}\vec{x}})>0$. Then $f$ is radially unbounded, diverging at least quadratically with $|\vec{x}|$. This proves that $f$ has a minimum. Furthermore $f$ is strictly convex, so the minimum is unique.  The second lemma shows that the gradient of a function diverges at least linearly, provided that its extremal eigenvalue is of the appropriate sign.
\begin{lemma}\label{lem:LinearBound}
If $\lambda_{inf}(f_{\vec{x}\vec{x}})>0$, then
$|f_\vec{x}(\vec{x})-f_\vec{x}(\vec{a})|\geq \lambda_{inf}(f_{\vec{x}\vec{x}})|\vec{x}-\vec{a}|
$.
\end{lemma}
\begin{proof}
Switching the roles of $\vec{x}$ and $\vec{a}$ in Lemma \ref{lem:QuadraticBound} yields
$f(\vec{a}) \geq f(\vec{x}) + f_{\vec{x}}(\vec{x})^T(\vec{a}-\vec{x}) + 
\lambda_{inf}(f_{\vec{x}\vec{x}})|\vec{x}-\vec{a}|^2/2$. Add this to the original inequality of Lemma
\ref{lem:QuadraticBound} to obtain
$
[f_{\vec{x}}(\vec{x})-f_{\vec{x}}(\vec{a})]^T(\vec{x}-\vec{a})\geq
\lambda_{inf}(f_{\vec{x}\vec{x}})|\vec{x}-\vec{a}|^2$
Use the Cauchy-Schwarz inequality, and divide by $|\vec{x}-\vec{a}|$. \end{proof}

The final lemma bounds a function using the magnitude of its gradient.
\begin{lemma}\label{lem:FunctionGrad}
Suppose that $\lambda_{inf}(f_{\vec{x}\vec{x}})>0$, and let $f_{min}$ be the minimal value of $f$. Then 
$f(\vec{a}) 
\leq f_{min}
+|f_\vec{x}(\vec{a})|^2/(2\lambda_{inf}(f_{\vec{x}\vec{x}}))$.
\end{lemma}
\begin{proof}
Lemma \ref{lem:QuadraticBound} implies
$f(\vec{x}) \geq f(\vec{a})+\min_\vec{b} \left\{
f_\vec{x}(\vec{a})^T\vec{b} + 
\lambda_{inf}(f_{\vec{x}\vec{x}})|\vec{b}|^2/2\right\}
=
f(\vec{a})-|f_\vec{x}(\vec{a})|^2/(2\lambda_{inf}(f_{\vec{x}\vec{x}})) 
$
Since the inequality holds for all $\vec{x}$ and $\vec{a}$, the left hand side can be replaced by its minimum, which is guaranteed to exist by the condition $\lambda_{inf}(f_{\vec{x}\vec{x}})>0$.
\end{proof}

\begin{proof}[Proof of Theorem \ref{thm:MainTheorem}]
  Boundedness of trajectories will be shown using the same Lyapunov function as   in Theorem \ref{thm:BoundedTrajectory}, with a specific choice of the   parameter $r$. The goal of the proof is to show that the Lyapunov function is radially unbounded.
The payoff function is the major problem here, because it has neither lower or upper bound. It could potentially diverge to $-\infty$ or $+\infty$ in different directions. (Either kind of divergence could be harmful, depending on the sign of $r$). On the other hand, the kinetic energy is nonnegative. The basic idea of the proof is to show that the growth of the kinetic energy more than balances out any negative divergence of $-rS$.

In particular, we'll consider the first case of the Theorem, $\lambda_{inf}(S_{\vec{x}\vec{x}})>0$, and show that $T^x-rS$ for an appropriate $r$ can be lower bounded by $-V(\vec{y})$, which is radially unbounded.

By the condition $\lambda_{inf}(S_{\vec{x}\vec{x}})>0$ and Lemma \ref{lem:QuadraticBound}, $S(\vec{x},\vec{y})\to \infty$ as $|\vec{x}|\to\infty$. Therefore $V(\vec{y})=\min_{\vec{x}}S(\vec{x},\vec{y})$ is well-defined for all $\vec{y}$. Now apply Lemma \ref{lem:FunctionGrad} to $f(\vec{x})=S(\vec{x},\vec{y})$, yielding
\begin{equation}\label{eq:UBound}
S(\vec{x},\vec{y})\leq V(\vec{y})+\frac{|S_\vec{x}(\vec{x},\vec{y})|^2}{2\lambda_{inf}(S_{\vec{x}\vec{x}})}
\end{equation}

Now use the Lyapunov function   $L=T-rS$ with $r=\gamma\lambda_{inf}(S_{\vec{x}\vec{x}})$, where $0<\gamma<1$   and   $\lambda_{sup}(S_{\vec{y}\vec{y}})<\gamma\lambda_{inf}(S_{\vec{x}\vec{x}})<\lambda_{inf}(S_{\vec{x}\vec{x}})$.   Such a $\gamma$ is guaranteed to exist, since   $\lambda_{sup}(S_{\vec{y}\vec{y}})<\lambda_{inf}(S_{\vec{x}\vec{x}})$ and   $\lambda_{inf}(S_{\vec{x}\vec{x}})>0$.
By the proof of Theorem \ref{thm:BoundedTrajectory},
\[
L=\half |S_{\vec{x}}|^2 + \half |S_{\vec{y}}|^2 -
\gamma \lambda_{inf}(S_{\vec{x}\vec{x}}) S
\]
is a Lyapunov function. Substituting the inequality \eqref{eq:UBound} yields
\[
L \geq \half (1-\gamma) |S_\vec{x}|^2 + \half|S_{\vec{y}}|^2 - \gamma
\lambda_{inf}(S_{\vec{x}\vec{x}})V(\vec{y})
\]
Since $-V(\vec{y})$ is radially unbounded by assumption, $L\to\infty$  if $|\vec{y}|\to\infty$. This is almost a proof that $L$ is radially unbounded.

To complete the proof, consider the behavior of $L$ if $|\vec{x}|\to\infty$ while $|\vec{y}|$ stays bounded.  By Lemma \ref{lem:LinearBound}, $|S_\vec{x}(\vec{x},\vec{y})-S_\vec{x}(\vec{0},\vec{y})| \geq \lambda_{inf}(S_{\vec{x}\vec{x}}) |\vec{x}|^2 $.  If $\vec{y}$ is bounded, so also is $S_\vec{x}(\vec{0},\vec{y})$. Therefore the inequality implies that $|S_{\vec{x}}|\to\infty$ as $|\vec{x}|\to\infty$ while $|\vec{y}|$ stays bounded.

The above arguments show that $L$ is radially unbounded in $(\vec{x},\vec{y})$, i.e., $L\to\infty$ as $|\vec{x}|^2+|\vec{y}|^2\to\infty$. Boundedness of trajectories follows.  Similar arguments can be made to prove the second case of Theorem 2.
\end{proof}

\begin{proof}[Proof of Corollary \ref{cor:Kinetic2}]
Since  $\lambda_{sup}(S_{\vec{y}\vec{y}})<0$, it only remains to be shown that   $U(\vec{x})=\max_\vec{y} S(\vec{x},\vec{y})$ is radially unbounded, and then  the second case of Theorem \ref{thm:MainTheorem} can be applied. It turns out that this follows   from the condition $\lambda_{inf}(S_{\vec{x}\vec{x}})>0$, as shown below.
By Lemma \ref{lem:QuadraticBound},
$U(\vec{x}) \geq U(\vec{0}) + U_\vec{x}(\vec{0})^T\vec{x}+\lambda_{inf}(U_{\vec{x}\vec{x}})|\vec{x}|^2/2
$.
The Hessian of $U$ is $U_{\vec{x}\vec{x}} = S_{\vec{x}\vec{x}}-
S_{\vec{x}\vec{y}}S_{\vec{y}\vec{y}}^{-1}S_{\vec{y}\vec{x}}$, where
the right hand side is evaluated at $(\vec{x},\vec{y}^\ast(\vec{x}))$
and $\vec{y}^\ast(\vec{x})=\argmax_\vec{y}S(\vec{x},\vec{y})$.
But $S_{\vec{y}\vec{y}}$ is negative definite, so that $S_{\vec{x}\vec{y}}S_{\vec{y}\vec{y}}^{-1}S_{\vec{y}\vec{x}}$
is negative definite.
This means that $\lambda_{inf}(U_{\vec{x}\vec{x}})\geq
\lambda_{inf}(S_{\vec{x}\vec{x}})$, yielding the lower bound
$U(\vec{x}) \geq U(\vec{0}) + U_\vec{x}(\vec{0})^T\vec{x} +  \lambda_{inf}(S_{\vec{x}\vec{x}})|\vec{x}|^2/2
$.
Therefore the inequality $\lambda_{inf}(S_{\vec{x}\vec{x}})>0$ implies that $U(\vec{x})$ is radially unbounded, and Theorem \ref{thm:MainTheorem} can be applied to prove boundedness.
\end{proof}

\section{Discussion}
Kose and Uzawa used the Euclidean distance from the steady state as a Lyapunov function for gradient descent-ascent. Their construction is applicable when the payoff function is convex-concave. In this case, no more than a single steady state may exist.

This paper introduced new Lyapunov functions for gradient descent-ascent. For the convex-concave case, the kinetic energy was used as a Lyapunov function. This led to Corollary \ref{cor:Kinetic}, which is almost identical to the Kose-Uzawa Theorem.  More generally, a combination of the kinetic energy and the payoff function was used as a Lyapunov function. This led to Theorems \ref{thm:BoundedTrajectory} and \ref{thm:MainTheorem}, which are more powerful than the Kose-Uzawa theorem. This added power was demonstrated by two examples, one linear and the other nonlinear. Without the convex-concave assumption, more than one steady state may exist.

The conditions of Theorem \ref{thm:MainTheorem} guarantee the existence of a minimax equilibrium of the game. While these conditions also assure that gradient descent-ascent will converge to a steady state, it need not be a minimax equilibrium. Therefore, gradient descent-ascent is a method of identifying candidate equilibria.

If the speeds of ascent and descent are made extremely different, then gradient descent-ascent is typically expected to converge to a local minimax equilibrium. For example, if the differential equation for $\vec{y}$ were replaced by $\tau\dot{\vec{y}}=S_{\vec{y}}$, then $\vec{y}$ would track $\argmax_{\vec{y}}S(\vec{x},\vec{y})$ instantaneously, in the limit as $\tau\to 0$. Therefore the $\vec{x}$ dynamics would generically converge to a local minimum of $U(\vec{x})$.

The relative convexity condition \eqref{eq:Convergence} implies that at least one of the inequalities $\lambda_{sup}(S_{\vec{y}\vec{y}}) < 0$ or $\lambda_{inf}(S_{\vec{x}\vec{x}})>0$ must hold. In other words, the payoff function must be strictly concave in $\vec{y}$ or strictly convex in $\vec{x}$. The Kose-Uzawa Theorem requires that both of these statements be true, while the present theorems require only one. It is an interesting open question whether general convergence results can be established for the case that neither statement is true.

\section*{Acknowledgments}
This work was stimulated by a collaboration with John Hopfield, Jeff Lagarias, and Tom Richardson on the dynamics of neural networks.  I have benefited from conversations with Sam Roweis.

\bibliography{gradient5}

\begin{thebibliography}{10}

\bibitem{kose1956ssv}
T.~Kose.
\newblock {Solutions of Saddle Value Problems by Differential Equations}.
\newblock {\em Econometrica}, 24:59--70, 1956.

\bibitem{Uzawa58}
H.~Uzawa.
\newblock Gradient method for concave programming, {II} global stability in the
  strictly concave case.
\newblock In K.~J. Arrow, L.~Hurwicz, and H.~Uzawa, editors, {\em Studies in
  linear and non-linear programming}, pages 127--132. Stanford Univ., 1958.

\bibitem{arrow1960sgp}
K.J. Arrow and L.~Hurwicz.
\newblock {Stability of the Gradient Process in n-Person Games}.
\newblock {\em Journal of the Society for Industrial and Applied Mathematics},
  8(2):280--294, 1960.

\bibitem{singh2000ncg}
S.~Singh, M.~Kearns, and Y.~Mansour.
\newblock {Nash convergence of gradient dynamics in general-sum games}.
\newblock {\em Proceedings of the Sixteenth Conference on Uncertainty in
  Artificial Intelligence}, pages 541--548, 2000.

\bibitem{ArrowHurwicz58gmc1}
K.~J. Arrow and L.~Hurwicz.
\newblock Gradient method for concave programming, {I} local results.
\newblock In K.~J. Arrow, L.~Hurwicz, and H.~Uzawa, editors, {\em Studies in
  linear and non-linear programming}, pages 117--126. Stanford Univ., 1958.

\bibitem{mjolsness1990ato}
E.~Mjolsness and C.~Garrett.
\newblock {Algebraic transformations of objective functions.}
\newblock {\em Neural Networks}, 3(6):651--669, 1990.

\bibitem{seung1998mah}
H.S. Seung, T.J. Richardson, JC~Lagarias, and J.J. Hopfield.
\newblock {Minimax and Hamiltonian dynamics of excitatory-inhibitory networks}.
\newblock {\em Advances in Neural Information Processing Systems}, 10, 1998.

\bibitem{gao2004nnc}
X.B. Gao, L.Z. Liao, and W.~Xue.
\newblock {A neural network for a class of convex quadratic minimax problems
  with constraints}.
\newblock {\em IEEE Trans. Neural Networks}, 15(3):622--628, 2004.

\bibitem{kerner1997chs}
E.H. Kerner.
\newblock {Comment on Hamiltonian structures for the n-dimensional
  Lotka--Volterra equations}.
\newblock {\em Journal of Mathematical Physics}, 38(2):1218, 1997.

\bibitem{samuelson1974bla}
P.A. Samuelson.
\newblock {A Biological Least-Action Principle for the Ecological Model of
  Volterra-Lotka}.
\newblock {\em Proceedings of the National Academy of Sciences},
  71(8):3041--3044, 1974.

\bibitem{logemann2004abn}
H.~Logemann and E.P. Ryan.
\newblock {Asymptotic Behaviour of Nonlinear Systems}.
\newblock {\em American Mathematical Monthly}, 111:864--889, 2004.

\bibitem{slotine1991anc}
J.J.E. Slotine and W.~Li.
\newblock {\em {Applied nonlinear control}}.
\newblock Prentice Hall, 1991.

\bibitem{lasalle1960eas}
J.~P. LaSalle.
\newblock {The Extent of Asymptotic Stability}.
\newblock {\em Proceedings of the National Academy of Sciences}, 46:363, 1960.

\bibitem{curry1944msd}
H.B. Curry.
\newblock {The method of steepest descent for nonlinear minimization problems}.
\newblock {\em Quart. Appl. Math}, 2:250--261, 1944.

\bibitem{absil16cid}
PA~Absil, R.~Mahony, and B.~Andrews.
\newblock {Convergence of the Iterates of Descent Methods for Analytic Cost
  Functions}.
\newblock {\em SIAM J. Optim.}, 16:531--547, 2005.

\bibitem{hofbauer1996edb}
J.~Hofbauer.
\newblock {Evolutionary dynamics for bimatrix games: A Hamiltonian system?}
\newblock {\em Journal of Mathematical Biology}, 34(5):675--688, 1996.

\bibitem{hopfield1986cnc}
JJ~Hopfield and DW~Tank.
\newblock {Computing with neural circuits: a model}.
\newblock {\em Science}, 233(4764):625, 1986.

\bibitem{Platt87}
J.~C. Platt and A.~H. Barr.
\newblock Constrained differential optimization for neural networks.
\newblock Technical Report Caltech-CS-TR-88-17, Caltech, 1987.

\bibitem{guckenheimer2002nod}
J.~Guckenheimer and P.~Holmes.
\newblock {\em {Nonlinear Oscillations, Dynamical Systems, and Bifurcations of
  Vector Fields}}.
\newblock Springer, 2002.

\end{thebibliography}
\bibliographystyle{unsrt}

\end{document}